\documentclass[12pt]{amsart}

\usepackage[margin = 1.2in]{geometry}
\usepackage{xcolor}

\makeatletter
\def\mathcolor#1#{\@mathcolor{#1}}
\def\@mathcolor#1#2#3{
  \protect\leavevmode
  \begingroup\color#1{#2}#3\endgroup
}
\makeatother

\usepackage{algorithm}
\usepackage{algpseudocode}

\newcounter{Halgorithm}
\floatevery{algorithm}{\stepcounter{Halgorithm}}

\usepackage{amsmath,amsthm,amssymb,amsfonts}
\usepackage{mathtools}
\usepackage{bm} 
\usepackage{upgreek} 
\usepackage{pifont} 
\usepackage{blkarray} 
\usepackage[all]{xy} 
\usepackage{enumitem} 
\usepackage{subdepth} 
\usepackage{fouridx} 
\usepackage{todonotes}
\usepackage{graphicx}
\usepackage[mode=build]{standalone}
\usepackage{tikz}
\usepackage{calc}
\usetikzlibrary{calc}
\usetikzlibrary{positioning}
\usetikzlibrary{decorations.pathmorphing,decorations.markings}
\usetikzlibrary{arrows.meta}
\usetikzlibrary{backgrounds}
\usetikzlibrary{fit}

\usepackage[colorlinks,linkcolor=red,anchorcolor=blue,citecolor=green]{hyperref}
\usepackage[capitalise,noabbrev]{cleveref}

\newtheorem{theorem}{Theorem}

\newtheorem{corollary}[theorem]{Corollary}

\newtheorem{example}[theorem]{Example}
\newtheorem{lemma}[theorem]{Lemma}

\newtheorem{remark}[theorem]{Remark}
\newtheorem{question}[theorem]{Question}
\newtheorem{proposition}[theorem]{Proposition}

\DeclareMathOperator{\dist}{dist}
\DeclareMathOperator{\Gram}{Gr}
\DeclareMathOperator{\Gal}{Gal}
\DeclareMathOperator{\diag}{diag}

\DeclarePairedDelimiter{\set}{\{}{\}}
\DeclarePairedDelimiter{\abs}{\vert}{\vert}

\DeclarePairedDelimiter{\lrangle}{\langle}{\rangle}

\newcommand{\cplxconj}[1]{\mathop{\overline{#1}}}

\begin{document}

\title{Classification of partially metric Q-polynomial association schemes with $m_{\mathtt 1} = 4$}

\author{Da Zhao}
\address{School of Mathematical Sciences, Shanghai Jiao Tong University, Shanghai, China.}
\email[Da Zhao]{jasonzd@sjtu.edu.cn}



\begin{abstract}
We classify the Q-polynomial association schemes with $m_{1} = 4$ which are partially metric with respect to the nearest neighbourhood relation. 
An association scheme is partially metric with respect to a relation $R_1$ if the scheme graph of $R_2$ is exactly the distance-2 graph of the scheme graph of $R_1$ under a certain ordering of the relations.

\end{abstract}

\maketitle

\section{Background}

The classification problem plays an import role in the research of association scheme. 
Hanaki and Miyamoto give a list of association schemes up to 30 vertices \cite{MR1972022}. 
There are studies on association schemes of small degree including \cite{MR1701285,MR2755088}. 
Bannai and Bannai finish the classification of primitive association schemes with multiplicity three \cite{MR2212140}. 
The Terwilliger algebra rises from the classification project of P and Q-polynomial association schemes. 
The classification of P-polynomial association schemes (distance-regular graphs) of valency three or four is done in \cite{MR850954,MR1701280}. 
Though Bannai-Ito conjecture and its dual, proved by Bang-Dubickas-Koolen-Moulton\cite{MR3281131} and Martin-Williford\cite{MR2494443}, guarantees the finiteness of P-polynomial (Q-polynomial) association schemes of given valency (multiplicity), the classification of Q-polynomial association schemes of small multiplicity is still yet undone. 
Bannai and the author revisit the paper \cite{MR2212140} and obtain the classification of Q-polynomial association schemes with $m_{\mathtt 1} = 3$ in \cite{bannai2017spherical}. 
We aim to finish the classification of Q-polynomial association schemes with $m_{\mathtt 1} = 4$.

This paper is organized as follows. In \cref{sec:Embossing}, we introduce association schemes and related concepts. In \cref{sec:Cross-stitch}, we state our main result.
The tools we used are exhibited in \cref{sec:derailment} and the proof of the main theorem is given in \cref{sec:Crocus}.

\section{Preliminary} \label{sec:Embossing}

\subsection{Graphs}

	A (simple) \emph{graph} $\Gamma$ consists of vertices $V$ and edges $E \subseteq \binom{V}{2}$. 
	Every graph in this paper is simple unless specified otherwise. 
	For a graph $\Gamma$, its complementary graph $\overline{\Gamma}$ is the graph with the same vertex set as $\Gamma$ and there is an edge $uv$ in $\overline{\Gamma}$ if and only if $uv$ is not an edge in $\Gamma$.
	A \emph{walk} on the graph $\Gamma$ (from vertex $x$ to vertex $y$ of length $\ell$) is a sequence of vertices $x=v_0, v_1, \ldots v_{\ell} = y$ such that $v_{i-1}v_i$ is an edge for $i = 1,2, \ldots, \ell$. 
	The \emph{distance} $\dist(x,y)$ between two vertices $x$ and $y$ is the length of the shortest walk from $x$ to $y$ (if there are no such walks, then the distance is $\infty$). 
	The \emph{diameter} $D(\Gamma)$ of a graph $\Gamma$ is the maximum distance between vertices in $\Gamma$. 
	For a graph $\Gamma$ of diameter $D$, its distance-$i$ graph $\Gamma_i$ ($1 \leq i \leq D$) is the graph which shares the vertices with $\Gamma$ and whose edges are the pairs of vertices of distance $i$ in $\Gamma$. 
	The set of vertices at distance $i$ from $x$ is denoted by $\Gamma_i(x)$. 
	Fix a base vertex $x$, we have a partition of the vertex set with respect to the distance of a vertex to $x$. 
	Given $y \in \Gamma_i(x)$, we define $c_i(x,y) = \abs{\Gamma_{i-1}(x) \cap \Gamma(y)}$, $a_i(x,y) = \abs{\Gamma_{i}(x) \cap \Gamma(y)}$, and $b_i(x,y) = \abs{\Gamma_{i+1}(x) \cap \Gamma(y)}$. 
	These numbers are useful to describe the structure of the graph.

	The \emph{adjacency matrix} $A = A(\Gamma) = (a_{uv})$ of a graph $\Gamma$ is a square matrix indexed by the vertex set $V$, where $a_{uv}$ is equal to $1$ if $uv$ is an edge and $0$ otherwise.
	Given the adjacency matrix of a graph $\Gamma$, the adjacency matrix of its complementary graph $\overline{\Gamma}$ is equal to $\overline{A} = J - I - A$, where $J$ is the all-one matrix and $I$ is the identity matrix.

\subsection{Association schemes}

		Let $X$ be a finite set. 
		A (symmetric) \emph{association scheme} $\mathfrak{X}$ on $X$ with $d$ classes is a pair $(X,\mathcal{R})$ such that 
		\begin{enumerate}
			\item $\mathcal{R} = \set{R_0, R_1, \ldots, R_d}$ is a partition of $X \times X$;
			\item $R_0 = \set{(x,x) : x \in X}$;
			\item $R_i = R_i^T := \set{(y,x) : (x,y) \in R_i}$ for each $0 \leq i \leq d$;
			\item For every pair $(x,y) \in R_h$, there are exactly $p_{ij}^h$ elements $z \in X$ such that $(x,z) \in R_i$ and $(z,y) \in R_j$, where $p_{ij}^h$ is a constant only depends on $i,j$ and $h$.
		\end{enumerate}

	The \emph{size} of the association scheme $\mathfrak{X}$ is the cardinality of $X$ and the \emph{degree} of the association scheme is $d$. 
	The elements of $\mathcal{R}$ are called the \emph{relations} of the association scheme, among which $R_0$ is called the trivial relation. 
	Every non-trivial relation $R_i$ ($i>0$) can be regarded as a graph with vertex set $X$ and two vertices $x$ and $y$ are adjacent if and only if $(x,y) \in R_i$. 
	We call this graph the \emph{scheme graph} of $R_i$ and call its adjacency matrix $A_i$ the relation matrix ($A_0=I$ by convention, though the graph is not simple). 
	Note that we do NOT use the notation $\Gamma_i$ for the scheme graph of $R_i$ to avoid confusion with the distance-$i$ graph of a graph $\Gamma$. 
	By using the relation matrices, the conditions in the definition of association scheme are equivalent to the following conditions. 
	\begin{enumerate}
		\item $\sum_{i=0}^d A_i = J$;
		\item $A_0 = I$;
		\item $A_i = A_i^T$;
		\item $A_i A_j = \sum_{h=0}^d p_{ij}^h A_h$.
	\end{enumerate}

	There are more general definitions of association scheme. 
	Without explanation, we simply exhibit them as follows: symmetric association schemes $\subsetneq$ commutative association schemes $\subsetneq$ (general) association schemes $=$ homogeneous coherent configurations $\subsetneq$ coherent configurations. 
	We will focus on symmetric association schemes in this paper. 
	The reader is referred to \cite{MR325420,MR398868} for definitions and related properties of coherent configurations. 

	A symmetric association scheme is called \emph{primitive} if the scheme graph of every non-trivial relation $R_i$ $(i > 0)$ is a connected graph.

	Let $\mathfrak X=(X,\{R_i\}_{0\leq i\leq d})$ be a symmetric association scheme. 
	Then it is automatically a commutative association scheme. 
	The relation algebra $\mathfrak{A} = \lrangle{A_0, A_1, \ldots, A_d}$ is commutative, hence it has another basis called primitive idempotents.  
	We denote by $E_{\mathtt i}$ the $i$-th primitive idempotent of the association scheme. 
	The `transition matrices' between the two basis of the Bose-Mesner algebra $\mathfrak{A} = \lrangle{A_0, A_1, \ldots, A_d} = \lrangle{E_{\mathtt 0}, E_{\mathtt 1}, \ldots, E_{\mathtt d}}$ are denoted by $P$ the first eigenmatrix and $Q$ the second eigenmatrix. 
	Formally $A_i = \sum_{\mathtt {j=0}}^{\mathtt d} P_{{\mathtt j}i} E_{\mathtt j}$ and $E_{\mathtt i} = \frac{1}{\abs{X}} \sum_{j=0}^d Q_{j \mathtt i} A_j$. 
	The intersection numbers $p_{ij}^k$ are given by $A_i A_j = \sum_{k=0}^d p_{ij}^k A_k$ and the Krein parameters $q_{ij}^k$ are given by $E_{\mathtt i} \circ E_{\mathtt j} = \frac{1}{\abs{X}} \sum_{\mathtt{k=0}}^{\mathtt d} q_{\mathtt{ij}}^{\mathtt k} E_{\mathtt k}.$ 
	Here $M \circ N$ represents the entrywise product, also known as the Hadamard product, of the matrices $M$ and $N$.
	Let $k_i = p_{ii}^0$ be the valencies of $\mathfrak X$ and let $m_{\mathtt i} = q_{\mathtt{ii}}^{\mathtt 0}$ be the multiplicities of $\mathfrak X$.

	For each primitive idempotent $E_{\mathtt i}$, the mapping $X \rightarrow \mathbb R^{m_{\mathtt i}}$ defined by
	\[x\rightarrow \overline{x}=\sqrt{\frac{|X|}{m_{\mathtt i}}}E_{\mathtt i}\phi_{x},\]
	gives a spherical representation of $\mathfrak{X}$ on the unit sphere $S^{m_{\mathtt i}-1}\subset \mathbb R^{m_{\mathtt i}}$, where $\phi_{x}$ is the characteristic vector of $x$ (regarded as a column vector). 
	The Gram matrix of this representation is exactly $\Gram = \frac{\abs{X}}{m_{\mathtt i}} E_{\mathtt i}$.
	In what follows, whenever the representation is faithful, we identify $\overline{X}$ and $X$ and call it a \emph{spherical embedding}. 

	For more information on association schemes, the reader may check \cite{MR882540} and \cite{MR1002568}.

\subsection{Partially metric (cometric) association schemes}

	A symmetric association scheme $\mathfrak X=(X,\{R_i\}_{0\leq i\leq d})$ is called \emph{$t$-partially metric} (with respect to a connected relation $R$) if, under a certain ordering of the relations, $A_i$ is a polynomial of degree $i$ in $A_1$ for $i=1,2,\ldots, t \leq d$, where $A$ is the relation matrix of $R_1 = R$. 
	In other words, the distance-$i$ graph of the scheme graph $R$ is exactly the scheme graph of $R_i$ for $i \leq t$. 
	A $d$-partially metric association scheme is called \emph{metric} or $P$-polynomial. 
	It is folklore that the scheme graph of $R_1$ in a metric association scheme is a distance-regular graph and vice versa. 
	If $\mathfrak X=(X,\{R_i\}_{0\leq i\leq d})$ is $t$-partially metric with respect to a connected relation $R$, then for $0 \leq i \leq t$ the numbers $a_i(x,y)$, $b_i(x,y)$ and $c_i(x,y)$ where $y \in R_i(x)$ are independent of the choice of $x$ and $y$. 
	Therefore we can use the notations $a_i$, $b_i$ and $c_i$ ($0 \leq i \leq t$) unambiguously. 
	Since every association scheme is trivially $1$-partially metric, we reserve partially metric association schemes for association schemes which are at least $2$-partially metric. 

	Similarly one can define ($t$-partially) cometric association schemes by replacing $A_i$ with $E_{\mathtt i}$ and usual matrix product with Hadamard product. 
	The cometric (also known as $Q$-polynomial) association schemes have similar properties as metric association schemes, but yet the combinatorial interpretation of cometric association scheme has not been revealed. 
	When talking about $Q$-polynomial association schemes, we always take a $Q$-polynomial ordering (may not be unique) of the primitive idempotents unless stated otherwise. 

	The representation of a Q-polynomial association scheme with respect to $E_{\mathtt 1}$ is always faithful, since the inner products between points of $X$, given by $\alpha_i = Q_1(i)/m_{\mathtt 1} = P_i(1)/k_i$, are different. 
	Among them there is a unique non-trivial relation whose inner product achieves maximum $\alpha_{\max} = \max \set{ \alpha_i : \alpha_i < 1}$. 
	We call it the \emph{nearest neighbourhood relation} and its scheme graph the \emph{nearest neighbourhood graph}.

\section{Main result} \label{sec:Cross-stitch}

Our main theorem is stated as follows.

\begin{theorem} \label{thm:Comprehension}
	Let $\mathfrak{X}$ be a Q-polynomial (symmetric) association scheme with $m_{\mathtt 1} = 4$. 
	Let $R$ be the nearest neighbourhood relation and $\Gamma$ be its scheme graph. 
	Suppose $\mathfrak{X}$ is partially metric with respect to $R$, then one of the following is true:
	\begin{enumerate}
		\item $\Gamma$ is the complete bipartite graph $K_{3,3}$ and $\mathfrak{X}$ is AS06[3];
		\item $\Gamma$ is the $16$-cell graph $K_{2,2,2,2}$ and $\mathfrak{X}$ is AS08[2]; 
		\item $\Gamma$ is the Rook graph $K_3 \square K_3 \cong L(K_{3,3})$ and $\mathfrak{X}$ is AS09[3];
		\item $\Gamma$ is the Johnson graph $J(5,2)$ and $\mathfrak{X}$ is AS10[3];
		\item $\Gamma$ is the crown graph $\overline{K_2 \square K_5}$ and $\mathfrak{X}$ is AS10[6];
		\item $\Gamma$ is the tesseract graph $Q_4$ and $\mathfrak{X}$ is AS16[30].
	\end{enumerate}
\end{theorem}

\begin{remark}
	Here the notation AS06[3] represents the 3rd association scheme of size 6 in Hanaki and Miyamoto's list of association schemes \cite{MR1972022}. 
	One can check all the Q-polynomial symmetric association schemes with $m_{\mathtt 1} = 4$ up to $30$ vertices in their list. 
	Only two of them are not listed in \cref{thm:Comprehension}. 
	They are AS05[1], where $\Gamma$ is the complete graph $K_5$, and AS24[43], where $\Gamma$ is the $1$-skeleton of the $24$-cell. 
\end{remark}

	\begin{table}[htbp]
		\begin{tabular}{|c|c|c|c|c|c|c|}
			\hline
			No. & $\abs{X}$ & $d$ & $k_1$ & Cosines & $\Gamma$ & $\Gamma(x)$\\ \hline
			1 & 5 & 1 & 4 & $(1,-\frac{1}{4})$ & $K_5$ & Complete graph $K_4$ \\
			2 & 6 & 2 & 3 & $(1,0,-\frac{1}{2})$ & $K_{3,3} $& Empty graph $N_3$ \\
			3 & 8 & 2 & 6 & $(1,0,-1)$ & $K_{2,2,2,2}$ & Octahedron graph \\
			4 & 9 & 2 & 4 & $(1,\frac{1}{4},-\frac{1}{2})$ & $K_3 \square K_3$ & perfect matching $2K_2$ \\
			5 & 10 & 2 & 6 & $(1,\frac{1}{6},-\frac{2}{3})$ & Johnson graph $J(5,2)$ & Rook graph $K_3 \square K_2$ \\
			6 & 10 & 3 & 4 & $(1,\frac{1}{4},-\frac{1}{4},-1)$ & Crown graph $\overline{K_2 \square K_5}$ & Empty graph $N_4$ \\
			7 & 16 & 4 & 4 & $(1,\frac{1}{2},0,-\frac{1}{2},-1)$ & Tesseract graph $Q_4$ & Empty graph $N_4$ \\
			8 & 24 & 4 & 8 & $(1,\frac{1}{2},0,-\frac{1}{2},-1)$ & 24-cell graph & Cube graph $H(3,2)$ \\ \hline
		\end{tabular}
		\caption{Q-polynomial association schemes with $m_{\mathtt 1} = 4$ up to $30$ vertices}
	\end{table}

\section{Tools} \label{sec:derailment}
	
\subsection{Spherical codes}

	A spherical code (on $S^{d-1}$) is nothing but a finite set $X$ on $S^{d-1} \subset \mathbb{R}^d$. 
	We call $X$ a spherical $A$-code if $X$ is a spherical code such that all inner products between distinct points of $X$ belong to $A$, formally $\set{ \lrangle{x,y}: x,y \in X, x \neq y} \subseteq A$. 
	Further we call $X$ a spherical $s$-code if $X$ is a spherical $A$-code for some $\abs{A} \leq s$. 
	In particular, the kissing number problem in $\mathbb{R}^d$ is equivalent to find the maximum cardinality of spherical $[-1,\frac{1}{2}]$-codes on $S^{d-1}$.
	
	\begin{theorem}[{\cite[Theorem 4.8]{MR0485471}}] \label{thm:progenity}
		Suppose $X$ is an $s$-code on $S^{d-1} \subset \mathbb{R}^d$, then $\abs{X} \leq \binom{d+s-1}{d-1} + \binom{d+s-2}{d-1}$.
	\end{theorem}

	We mainly need the bound for $d=3$ and $s=2$. 
	Note that the $1$-codes are exactly (vertices of) regular simplexes.

	\begin{example}
		Suppose $X$ is a $2$-code on $S^2 \subset \mathbb{R}^3$, then $\abs{X} \leq 9$. 
	\end{example}

	\begin{theorem}[{\cite{MR2415397}}] \label{thm:chorea}
		The kissing number $K(4)$ in $4$-dimensional Euclidean space is $24$. 
		In other words, suppose $X$ is a spherical $[-1,\frac{1}{2}]$-code on $S^3 \subset \mathbb{R}^4$, then $\abs{X} \leq 24$.
	\end{theorem}

	The following proposition follows from basic linear algebra.

	\begin{proposition} \label{prop:noncreative}
		Let $u_1, u_2, \ldots, u_n$ be $n$ vectors lie in the $\mathbb{R}^d$, and let $\Gram$ be the Gram matrix of these vectors. 
		Then the rank of $\Gram$ is at most $d$. 
		Consequently $0$ is an eigenvalue of $\Gram$ of multiplicity at least $n-d$. 
		Suppose the characteristic polynomial $\chi_{\Gram}(t) = \det (tI - \Gram)$ of $\Gram$ is equal to $h_0 + h_1 t + \cdots + t^n$, then $h_i = 0$ for $i= 0, 1, \ldots, n-d-1$.
	\end{proposition}

\subsection{Cosine sequence}
	The first eigenmatrix $P$ and the second eigenmatrix $Q$ of an association scheme is related by $M^{-1}Q^T = \cplxconj{P} K^{-1}$, where $K = \diag(k_0, k_1, \ldots, k_d)$ and $M = \diag(m_{\mathtt 0}, m_{\mathtt 1}, \ldots, m_{\mathtt d})$.
	We follow Terwilliger by using the \emph{cosine matrix} $C=M^{-1}Q^T = \cplxconj{P} K^{-1}$ to complete our computation. 
	The $(i,\mathtt{j})$-entry $\omega_{i,\mathtt{j}}$ of the cosine matrix $C$ is equal to the cosine of the angle between two points of relation $R_i$ in the spherical representation with respect to $E_{\mathtt j}$. 
	We call the numbers $\omega_{i} = \omega_{i,\mathtt{j}}$ the \emph{cosine sequence} corresponding to $E_{\mathtt j}$.
	The indices of the relations and those of primitive idempotents are sometimes confusing. 
	Note that we use alphabets in normal math font ($i,j,k, \ldots$) for indices of the relations and distinguish the indices of the primitive idempotents by using alphabets in typewriter font ($\mathtt{i,j,k, \ldots}$). 
	The identities $A_i A_j = \sum_{h=0}^d p_{ij}^h A_h$ and $E_{\mathtt{i}} \circ E_{\mathtt{j}} = \frac{1}{\abs{X}} \sum_{\mathtt{h=0}}^\mathtt{d} q_{\mathtt{ij}}^{\mathtt{h}} E_{\mathtt{h}}$ can be written as $k_i \omega_{i,-} \omega_{j,-} = \sum_{h=0}^d p_{hi}^j \omega_{h,-}$ and $m_{\mathtt{i}} \omega_{-,\mathtt{i}} \omega_{-,\mathtt{j}} = \sum_{\mathtt{h=0}}^\mathtt{d} q_{\mathtt{hi}}^{\mathtt{j}} \omega_{-,\mathtt{h}}$ respectively.
\subsection{Splitting field} 
	
	Given an association scheme $\mathfrak{X}$, its \emph{splitting field} $K$ is the smallest field containing all the eigenvalues of $\mathfrak{X}$. 
	It is conjectured by Bannai and Ito \cite[p.123]{MR882540} that the splitting field of a commutative association scheme is contained in a cyclotomic field. 
	For $Q$-polynomial association schemes, there are stronger restrictions on the splitting field. 
	\begin{theorem} [{\cite[Theorem 2.2]{MR2494443}}] \label{thm:massageuse}
		The splitting field of any $Q$-polynomial association scheme with $m_{\mathtt 1} > 2$ is at most a degree two extension of the rationals.
	\end{theorem}

	This follows from the characterization of the association schemes with multiple $Q$-polynomial structures. 
	\begin{theorem} [{\cite[Theorem 1]{MR1609873} \cite[Theorem 3]{MR3134270}}] \label{thm:pumice}
		Let $\mathfrak{X} = (X, \set{R_i}_{0 \leq i \leq d})$ be a $Q$-polynomial association scheme with respect to the ordering $E_{\mathtt 0}, E_{\mathtt 1}, \ldots, E_{\mathtt d}$ of the primitive idempotents and $m_{\mathtt 1} > 2$. 
		\begin{enumerate}
			\item Suppose $\mathfrak{X}$ is $Q$-polynomial with respect to another ordering. Then the new ordering is one of the following. 
			\begin{enumerate}
				\item $E_{\mathtt 0}, E_{\mathtt 2}, E_{\mathtt 4}, E_{\mathtt 6}, \ldots, E_{\mathtt 5}, E_{\mathtt 3}, E_{\mathtt 1}$;
				\item $E_{\mathtt 0}, E_{\mathtt d}, E_{\mathtt 1}, E_{\mathtt{d-1}}, E_{\mathtt 2}, E_{\mathtt{d-2}}, E_{\mathtt 3}, E_{\mathtt{d-3}}, \ldots$;
				\item \label{itm:Albian} $E_{\mathtt 0}, E_{\mathtt d}, E_{\mathtt 2}, E_{\mathtt{d-2}}, E_{\mathtt 4}, E_{\mathtt{d-4}}, \ldots, E_{\mathtt{d-5}}, E_{\mathtt 5}, E_{\mathtt{d-3}}, E_{\mathtt 3}, E_{\mathtt{d-1}}, E_{\mathtt 1}$;
				\item \label{itm:halter}$E_{\mathtt 0}, E_{\mathtt{d-1}}, E_{\mathtt 2}, E_{\mathtt{d-3}}, E_{\mathtt 4}, E_{\mathtt{d-5}}, \ldots, E_{\mathtt 5}, E_{\mathtt{d-4}}, E_{\mathtt 3}, E_{\mathtt{d-2}}, E_{\mathtt 1}, E_{\mathtt d}$.
			\end{enumerate}
			\item $\mathfrak{X}$ has at most two $Q$-polynomial structures.
		\end{enumerate}
	\end{theorem}

	The following two theorems by William Martin and Jason Williford are crucial to our method.

	\begin{theorem}[\cite{2017BITMartin}] \label{thm:macrocarpous}
		For every $Q$-polynomial association scheme $\mathfrak{X}$ with $m_{\mathtt 1} > 2$, the entries in the even columns (indexed by $E_{\mathtt 0}, E_{\mathtt 2}, E_{\mathtt 4}, \ldots$) of its second eigenmatrix $Q$ are rational integers, and the entries in the odd columns (indexed by $E_{\mathtt 1}, E_{\mathtt 3}, E_{\mathtt 5}, \ldots$) of its second eigenmatrix $Q$ are algebraic integers of a quadratic field. 
	\end{theorem}

	\begin{proof}
		If the splitting field of association scheme is $\mathbb{Q}$, then the conclusion follows from the fact that the entries of the second eigenmatrix $Q$ are eigenvalues of adjacency matrices. 
		Otherwise by \cref{thm:massageuse} the splitting field is a quadratic field $K = \mathbb{Q}[\sqrt{p}]$. 
		The Galois group $\Gal(K/\mathbb{Q})$ acts naturally on the primitive idempotents. 
		The non-trivial element $\sigma$ of the Galois group induces another $Q$-polynomial ordering. 
		Since $\sigma$ is an involution, the new $Q$-polynomial ordering must be either case (c) or case (d) in \cref{thm:pumice}. 
		The theorem follows by noting that the primitive idempotents $E_{\mathtt 0}, E_{\mathtt 2}, E_{\mathtt 4}, \ldots$ are fixed by $\sigma$.
	\end{proof}

	From \cref{thm:macrocarpous}, one can deduce the following bound on the degree of a $Q$-polynomial association scheme by the valency of its nearest neighbourhood relation.

	\begin{theorem}[\cite{2017BITMartin}] \label{thm:grouped}
		Suppose $\mathfrak{X} = (X, \set{R_i}_{0 \leq i \leq d})$ is a $Q$-polynomial association scheme with respect to the ordering $E_{\mathtt 0}, E_{\mathtt 1}, \ldots, E_{\mathtt d}$ of the primitive idempotents and $m_{\mathtt 1} > 2$. 
		Let $v_1$ be the valency of its nearest neighbourhood relation in the spherical representation of $\mathfrak{X}$ with respect to $E_{\mathtt 1}$. 
		Then $d \leq 4v_1 + 1$. 
		Moreover, if the splitting field of $\mathfrak{X}$ is the rational field, then $d \leq 2v_1 +1$. 
	\end{theorem}

	We end this subsection by another theorem concerning the eigenvalues. 

	\begin{theorem}[{\cite[Theorem 4.4]{MR2494443}}] \label{thm:zymogenous}
		Let $K > 0$ and $S$ the set of all monic polynomials with degree $n$ over the integers, all of whose roots lie in $[-K,K]$. Then $S$ is finite.
	\end{theorem}

	\begin{corollary}[{\cite[Corollary 4.5]{MR2494443}}] \label{coro:resemblingly}
		Let $K > 0$ and let $n$ be a positive integer. Then there are only finitely many algebraic integers $a$ satisfying 
		\begin{enumerate}
			\item the minimal polynomial of $a$ over the rationals has degree at most $n$, and 
			\item $a$ and all of its algebraic conjugates lie in the interval $[-K,K]$.
		\end{enumerate}
	\end{corollary}

\subsection{Light tail}

	Let $\mathfrak{X} = (X, \set{R_i}_{0 \leq i \leq d})$ be a partially metric association scheme, and let $A_1$ be the adjacency matrix of $R_1$. 
	A primitive idempotent $E_{\mathtt{j}}$ is called a \emph{light tail} if the matrix $F := \sum_{\mathtt{h \neq 0}} q_{\mathtt{jj}}^{\mathtt{h}} E_{\mathtt{h}}$ is nonzero and $A_1F = \eta F$ for some real number $\eta$. 
	The following theorem characterizes when light tail exists.

	\begin{theorem}[{\cite[Theorem 3.2]{MR1002568} and \cite[Theorem 2.2]{MR3772733}}]
		Let $\mathfrak{X} = (X, \set{R_i}_{0 \leq i \leq d})$ be a partially metric association scheme of degree $d \geq 2$ and assume that the valency of the relation $R_1$ is at least three. 
		Let $E$ be a primitive idempotent with multiplicity $m$ for corresponding eigenvalue $\theta \neq \pm k$. Then
		\begin{equation}
			m \geq k - \frac{k(\theta + 1)^2 a_1 (a_1 + 1)}{((a_1 + 1)\theta + k)^2 + k a_1 b_1}
		\end{equation}
		with equality if and only if $E$ is a light tail.
	\end{theorem}
	Note that if the first primitive idempotent $E_\mathtt{1}$ of a $Q$-polynomial association scheme is a light tail, then we have $q_{\mathtt{11}}^{\mathtt{1}} = 0$.

\subsection{Relation-distribution diagram}

	Given a (commutative) association scheme $\mathfrak{X} = (X, \set{R_i}_{0 \leq i \leq d})$, for every relation $R_i$ of $\mathfrak{X}$ there is an associated \emph{relation-distribution diagram} $\Delta_i$ to $R_i$. 
	The relation-distribution diagram $\Delta_i$ takes all the relations of the association scheme as the vertices, and there is an arc of weight $p_{hi}^j$ from vertex $R_j$ to vertex $R_h$ whenever $p_{hi}^j > 0$. 
	Note that there could be loops and the weight of the arc from $R_j$ to $R_h$ may be different from that of the arc from $R_h$ to $R_j$. 
	Since $\sum_{h=0}^d p_{hi}^j = k_i$, the total weight of the out arcs from any vertex $R_j$ is a constant $k_i$. 
	In this paper, we mostly focus on the relation-distribution diagram associated to the first relation in the partially metric associated scheme.

	In \cite{MR1635554} Yamazaki proves a lemma putting restrictions on the structure of relation-distribution diagram, and later it is generalized in \cite{MR3772733}.

	\begin{lemma}[{\cite[Lemma 2.8]{MR1635554} \cite[Lemma 2.4]{MR3772733}}]
		Let $\mathfrak{X} = (X, \set{R_i}_{0 \leq i \leq d})$ be a (symmetric) associated scheme with degree $d \geq 3$ and a connected scheme graph $\Gamma$ of $R_1$. 
		Let $x, z$ be vertices such that $(x,z) \in R_2$ and $\dist(x,z) = i \geq 2$. 
		Assume that there exist two distinct neighbours $z_3, z_4$ of $z$ and two distinct relations $R_3, R_4$ such that $(x,z) \in R_3$, $(x,z) \in R_4$, $\dist(x, z_3) = \dist(x, z_4) = i + 1$, and $c_{i+1}(x,z_3) = 1$. 
		Then there exists a relation $R_5$ such that $p_{35}^1 \neq 0$, $p_{45}^1 \neq 0$, and $R_5 \cap \Gamma_i = \emptyset$, where $\Gamma_i$ is the distance-$i$ graph of $\Gamma$.
	\end{lemma}

\subsection{Algorithm to generate feasible relation-distribution diagrams}

Recall that the cosine sequences satisfy $k_i \omega_{i,-} \omega_{j,-} = \sum_{h=0}^d p_{hi}^j \omega_{h,-}$ and $m_{\mathtt{i}} \omega_{-,\mathtt{i}} \omega_{-,\mathtt{j}} = \sum_{\mathtt{h=0}}^\mathtt{d} q_{\mathtt{hi}}^{\mathtt{j}} \omega_{-,\mathtt{h}}$. 
These equations are identities on the cosines of adjacent vertices in the relation-distribution diagram and cosines of the same relation for different idempotents. 
We can use these to generate feasible relation-distribution diagrams together with the cosines step-by-step.

Before the algorithm is shown, let us prepare some necessary facts. 

\begin{proposition} \label{prop:cateress}
	Let $\mathfrak{X} = (X, \set{R_i}_{0 \leq i \leq d})$ be a Q-polynomial association scheme and $R_i$ a relation of valency $k_i$. 
	Then the number of possible cosines $\omega_{i,-}$ of $R_i$ is finite.
\end{proposition}

\begin{proof}
	Note that $\omega_{i,-} = \lambda / k_i$ where $\lambda$ is an eigenvalue of $R_i$. 
	The claim follows from \cref{coro:resemblingly}.
\end{proof}

\begin{proposition} \label{prop:thaumatrope}
	Let $\mathfrak{X} = (X, \set{R_i}_{0 \leq i \leq d})$ be a Q-polynomial association scheme with natural ordering of primitive idempotents. 
	Then $q_{\mathtt{01}}^{\mathtt{0}} = 0$, $q_{\mathtt{01}}^{\mathtt{1}} = 1$, $q_{\mathtt{11}}^{\mathtt{0}} = m_{\mathtt{1}}$, $q_{\mathtt{11}}^{\mathtt{1}} = \frac{(m_{\mathtt{1}} -1)\omega_{-,\mathtt{2}} - m_{\mathtt{1}} \omega_{-,\mathtt{1}}^2 + 1}{\omega_{-,\mathtt{2}} - \omega_{-,\mathtt{1}}}$ and $q_{\mathtt{21}}^{\mathtt{1}} = m_{\mathtt{1}} - 1 - q_{\mathtt{11}}^{\mathtt{1}}$.
\end{proposition}

\begin{proof}
	The first three identities hold for every symmetric association scheme. 
	The last two follow from $m_{\mathtt{1}} \omega_{-,\mathtt{1}} \omega_{-,\mathtt{1}} = q_{\mathtt{01}}^{\mathtt{1}} \omega_{-,\mathtt{0}} + q_{\mathtt{11}}^{\mathtt{1}} \omega_{-,\mathtt{1}} + q_{\mathtt{21}}^{\mathtt{1}} \omega_{-,\mathtt{2}}$ and $q_{\mathtt{01}}^{\mathtt{1}}+ q_{\mathtt{11}}^{\mathtt{1}} + q_{\mathtt{21}}^{\mathtt{1}} = m_{\mathtt{1}}$ in Q-polynomial association scheme.
\end{proof}

Now we exhibit the algorithm to generate feasible relation-distribution diagrams. 
In a nutshell the algorithm mechanizes the arguments in \cite{MR3772733}. 

\begin{algorithm} 
	\caption{Algorithm to generate feasible relation-distribution diagrams}\label{alg:mussel}
	\begin{algorithmic}[1]
        \Function{GenerateDiagram}{Diagram, ToDo, Cosines}
        	\If{ToDo $= \emptyset$}
        		\State{Print the diagram and the cosines}
        		\State{\Return}
        	\EndIf
        	\State{Current $\gets$ first of ToDo}
	        \State{NewToDo $\gets$ rest of ToDo}
        	\State{RemainValency $\gets k_1 - $ total weight of neighbours of Current in Diagram}
        	\ForAll{Arrangement of RemainValency} \Comment{See \cref{rmk:accusative} \cref{itm:yeorling}}
        		\State{NewDiagram $\gets$ compose Diagram with weighted arcs in Arrangement}
        		\State{NewToDo $\gets$ compose NewToDo with new relations in Arrangement}
        		\If{NewDiagram is not valid} \Comment{See \cref{rmk:accusative} \cref{itm:septentrionally}}
        			\State{\textbf{continue}}
        		\EndIf
        		\State{Solutions $\gets$ solutions of equations} \Comment{See \cref{rmk:accusative} \cref{itm:versicolorous}}
        		\If{solutions can not be determined}
        			\ForAll{possible value of the cosines} \Comment{See \cref{rmk:accusative} \cref{itm:saprolegnious}}
        				\State{NewCosines $\gets$ compose Cosines with this value}
        				\State{\Call{GenerateDiagram}{NewDiagram, ToDo, NewCosines}}
        			\EndFor
        		\Else
	        		\ForAll{Solution $\in$ Solutions}
	        			\If{Solution is not valid} \Comment{See \cref{rmk:accusative} \cref{itm:polygamical}}
		        			\State{\textbf{continue}}
		        		\EndIf
		        		\State{NewCosines $\gets$ compose Cosines with Solution}
		        		\State{\Call{GenerateDiagram}{NewDiagram, NewToDo, NewCosines}}
	        		\EndFor
	        	\EndIf
        	\EndFor
        \EndFunction
	\end{algorithmic}
\end{algorithm}

\begin{figure}
	\label{fig:hemistrumectomy}
	\centering
	\includegraphics[width=0.8\textwidth]{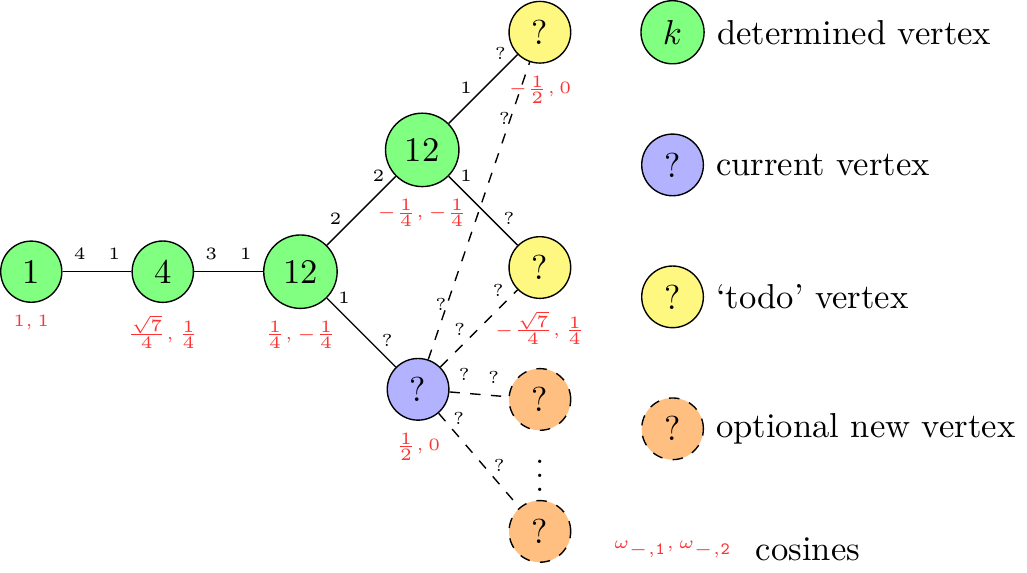}
	\caption{an intermediate step occurred when applying \cref{alg:mussel} to $\Gamma(x_0) = N_4$ in \cref{lem:automatograph}}
\end{figure}

\begin{proof}
	\cref{alg:mussel} generates feasible relation-distribution diagram of the nearest neighbourhood relation together with the first few (two is enough for the purpose of this paper) columns of the cosine matrix by iterative search with pruning. 
	The search follows the order of the breadth-first search of the diagram. 
	For each iteration, it deals with one vertex of the diagram, called \emph{Current}, which is the first among the list \emph{ToDo} of all undetermined vertices. 
	We exhaust all possible arrangements of the neighbours of the vertex \emph{Current} and the corresponding intersection numbers around this vertex. 
	There are finitely many such arrangements since the valency of the nearest neighbourhood relation is fixed. 
	The vertex may have neighbours which have not been exploited and we will add the new vertices to \emph{ToDo}.
	For each possible arrangement, we try to solve the cosines of the new vertices. 
	If there are not too many new vertices, then we have enough equations to determine the cosines. 
	Otherwise we can exhaust possible cosines with the help of \cref{prop:cateress}.
	Once the cosines are obtained, we can move on to the next step of iteration. 
	Finally the iteration depth is bounded by \cref{thm:grouped}. 
	Note that for each iteration, the cosines of every vertex (that has been exploited) in the diagram are known while the corresponding intersection numbers are known only for determined vertices but not those in the \emph{ToDo} list. 
	The pruning can be adopted whenever possible. 
	See the \cref{rmk:accusative} \cref{itm:septentrionally,itm:polygamical,itm:saprolegnious,itm:caliper} for detail. 
\end{proof}

\begin{remark} \label{rmk:accusative}
	Here we clarify some steps in the algorithm.
	\begin{enumerate}
		\item When considering the arrangement of the remaining valency, the weight can be distributed to arcs among existing vertices as well as new vertices. \label{itm:yeorling}
		\item For the purpose of this paper, we consider the equations $k_1 \omega_{1,-} \omega_{i,-} = \sum_{h=0}^d p_{h1}^i \omega_{h,-}$ where $-$ is $\mathtt{1}$ or $\mathtt{2}$, and $m_{\mathtt{1}} \omega_{-,\mathtt{1}} \omega_{-,\mathtt{1}} = \sum_{\mathtt{h=0}}^\mathtt{2} q_{\mathtt{h1}}^{\mathtt{1}} \omega_{-,\mathtt{h}}$ where $-$ ranges over all new relations. \label{itm:versicolorous}
		\item A diagram is valid if none of properties of relation-distribution diagram is violated. 
		In particular, it must satisfy Yamazaki's Lemma and the identity $k_h p_{ij}^h = k_i p_{jh}^i = k_j p_{hi}^j$. 
		\label{itm:septentrionally}
		\item There are finitely many possible solutions due to \cref{prop:cateress}. 
		Note that one can exhaust some cosine values and leave the rest to the equations. \label{itm:saprolegnious}
		\item A solution is valid if none of properties of the cosine sequences is violated. 
		In particular, the cosines are real and satisfy \cref{thm:massageuse}. 
		Moreover the cosine sequence corresponding to the first primitive idempotent $E_{\mathtt 1}$ of a $Q$-polynomial association scheme has distinct values. \label{itm:polygamical}
		\item It is useful to fix the splitting field at the very beginning. \label{itm:caliper}
		\item The algorithm can be initialized with any incomplete relation-distribution diagram together with cosines. 
		One can always start with two relations (relation $R_0$ and $R_1$). 
	\end{enumerate}
\end{remark}

\subsection{Local characterization of graphs}

We say a graph $G$ is locally a graph $H$ if for every vertex $v$ of $G$, the induced subgraph of $G$ on the neighbourhood of $v$ is isomorphic to $H$.
There are extensive studies on such object. 
Here we only need some simple cases.

\begin{theorem} [\cite{MR635704}] \label{thm:spindlelike}
	The only connected locally $C_5$ graphs is the 1-skeleton of the icosahedron.
\end{theorem}

\begin{theorem} [\cite{MR826793}] \label{thm:hotmouthed}
	The only connected locally octahedron graph is the 1-skeleton of the 4-dimensional hyperoctahedron.
\end{theorem}

\begin{proposition} \label{prop:integrally}
	If a connected graph $G$ is locally the Rook graph $K_3 \square K_2$, then $G$ is isomorphic to the Johnson graph $J(5,2)$. 
\end{proposition}

\begin{proof}
	Fix a vertex $x_0$ of $G$, we consider the first $H_1(x_0)$ and second neighbourhood $H_2(x_0)$ of $x_0$. 
	By assumption the first neighbourhood of $x_0$ induces a Rook graph $K_3 \square K_2$.
	Let $H_1(x_0) = \set{a,b,c,\overline{a},\overline{b},\overline{c}}$ be the first neighbourhood of $x_0$ such that $\set{a,b,c}$ and $\set{\overline{a},\overline{b},\overline{c}}$ are two triangles and that there are edges between $x$ and $\overline{x}$ where $x$ is $a,b$ or $c$. 
	For convenience we define $\overline{\overline{x}} := x$.

	Now we consider the neighbourhood of $a$, which should induce a Rook graph $K_3 \square K_2$ as well. 
	We've already got $4$ neighbours of $a$, which are $x_0, \overline{a},b$ and $c$. 
	Among them $\set{x_0,b,c}$ induces a triangle. 
	So the other two neighbours of $a$, say $u,v$ should form a triangle together with $\overline{a}$ and be adjacent of one of $b$ and $c$ each. 
	Let us name the vertices in $H_2(x_0)$ by the concatenation of its neighbours in $H_1(x_0)$. 
	Note that the above argument shows that for every vertex $w \in H_2(x_0)$, it is adjacent to $x \in H_1(x_0)$ if and only if it is adjacent to $\overline{x} \in H_1(x_0)$. 
	We may rename $u,v$ as $a\overline{a}b\overline{b}$ and $a\overline{a}c\overline{c}$.
	This tells us that every vertex in $H_2(x_0)$ has at least four neighbours in $H_1(x_0)$. 
	Let us do double counting on the edges between $H_1(x_0)$ and $H_2(x_0)$. 
	There are $6 \times 2$ edges. 
	Recall that there are at least two vertices in $H_2(x_0)$ and they are connected. 
	Hence there must be exactly three vertices in $H_2(x_0)$. 
	They form a triangle and each has exactly four neighbours in $H_1(x_0)$. 
	Consequently $H_2(x_0)$ has to be $\set{a\overline{a}b\overline{b}, a\overline{a}c\overline{c}, b\overline{b}c\overline{c}}$ and $G$ is the Johnson graph $J(5,2)$.
\end{proof}

\section{Proof of Main Theorem} \label{sec:Crocus}

In this section, $X$ is always the spherical embedding of a partially metric Q-polynomial association scheme with $m_{\mathtt 1} = 4$. 
We use $R_1$ to denote the partially metric nearest neighbourhood relation of $\mathfrak{X}$, and $\Gamma$ for its scheme graph. 
Naturally $R_2$ is reserved for the distance two relation of $R_1$. 
For a fixed point $x_0 \in X$, the induced subgraph on its neighbourhood $R_1(x_0)$ is denoted by $\Gamma(x_0)$. 
The points $R_1(x_0) \subset X$ are on an affine hyperplane, hence they are on a translated sphere $S^2$. 
For convenience, we say $R_1(x_0)$ is geometrically an object $O$ if the points $R_1(x_0)$ form a copy of the object $O$ in the spherical embedding $X$.

The proof of \cref{thm:Comprehension} is divided into two steps. 
First we classify all possible local structure of the spherical embedding. 
Secondly we determine the association schemes which carry the classified local structure. 
In some cases, \cref{alg:mussel} is used to exhaust feasible relation-distribution diagram associated to the nearest neighbourhood relation. 

\begin{lemma} \label{lem:automatograph}
	Under the assumptions of \cref{thm:Comprehension}, one of the following is true.
	\begin{enumerate}
		\item $R_1(x_0)$ is geometrically a triangle, and $\Gamma(x_0)$ is the empty graph $N_3$ or the complete graph $K_3$;
		\item $R_1(x_0)$ is geometrically a tetrahedron, and $\Gamma(x_0)$ is the empty graph $N_4$ or the complete graph $K_4$;
		\item $R_1(x_0)$ is geometrically a $2$-antiprism, and $\Gamma(x_0)$ is the perfect matching $2K_2$ or the $4$-cycle;
		\item $R_1(x_0)$ is geometrically a pentagon, and $\Gamma(x_0)$ is the $5$-cycle;
		\item $R_1(x_0)$ is geometrically a uniform $3$-prism, and $\Gamma(x_0)$ is the Rook graph $K_3 \square K_2$;
		\item $R_1(x_0)$ is geometrically an octahedron, and $\Gamma(x_0)$ is the octahedron graph.
	\end{enumerate}
\end{lemma}

\begin{proof}
	Let us consider the points $R_1(x_0)$ and the relations among these points. 
	Since the nearest neighbourhood graph $\Gamma$ is partially metric, there are at most two relations, besides the trivial relation $R_0$, among the points $R_1(x_0)$. 
	Note that the points $R_1(x_0)$ are on a translated sphere $S^2$. 
	So $R_1(x_0)$ forms a spherical $1$-code or spherical $2$-code on $S^2 \subset \mathbb{R}^3$. 
	By \cref{thm:progenity}, we get $\abs{R_1(x_0)} \leq 9$.
	The scheme graph of $R_1$ (resp. $R_2$) restricted to $R_1(x_0)$ is a regular graph of valency $p_{11}^1$ (resp. $p_{12}^1$). 
	The connected regular graphs of small order are classified and they can be generated by the program \textrm{GENREG}, see \cite{MR1665972}. 
	We denote by $B_1$ the adjacency matrix of $\Gamma(x_0)$. 
	Then the Gram matrix of $R_1(x_0)$ on $S^2$ is equal to $\Gram = I + \beta_1 B_1 + \beta_2 (J-I-B_1)$. 
	For each regular graph with at most $9$ vertices, we obtain many equations on $\beta_1$ and $\beta_2$ by \cref{prop:noncreative}. 
	Moreover we have $\beta_1 < \frac{1}{2}$ and $\beta_2 < \frac{1}{2}$ because $R_1$ is the nearest neighbourhood relation.
	Most regular graphs are excluded because the corresponding system of equations and inequalities has no solution. 
	The remaining graphs are exactly those listed in the statement of \cref{lem:automatograph}. 
\end{proof}

\begin{proof}[Proof of \cref{thm:Comprehension}]
	We consider each case in \cref{lem:automatograph} separately. 

	If $\Gamma(x_0)$ is the octahedron graph, then by \cref{thm:hotmouthed} the only connected locally octahedron graph is the $16$-cell graph. 
	The induced association scheme AS08[2] is Q-polynomial with first multiplicity $m_{\mathtt 1} = 4$. 

	If $\Gamma(x_0)$ is the Rook graph $K_3 \square K_2$, then by \cref{prop:integrally} $\Gamma$ is the Johnson graph $J(5,2)$. 
	The induced association scheme AS10[3] is Q-polynomial with first multiplicity $m_{\mathtt 1} = 4$. 

	If $\Gamma(x_0)$ is the $5$-cycle, then by \cref{thm:spindlelike} the only connected locally pentagon graph is the icosahedron graph. 
	The induced association scheme is Q-polynomial with first multiplicity $m_{\mathtt 1} = 3$. 

	If $\Gamma(x_0)$ is the $4$-cycle, it is straightforward to show that $\Gamma$ is the octahedron graph. 
	The induced association scheme is Q-polynomial with first multiplicity $m_{\mathtt 1} = 3$. 

	If $\Gamma(x_0)$ is the complete graph $K_4$, it is straightforward to show that $\Gamma$ is the complete graph $K_5$. 
	The induced association scheme is Q-polynomial with first multiplicity $m_{\mathtt 1} = 4$. 
	Please note that $K_5$ is not partially metric.

	If $\Gamma(x_0)$ is the complete graph $K_3$, it is straightforward to show that $\Gamma$ is the complete graph $K_4$. 
	The induced association scheme is Q-polynomial with first multiplicity $m_{\mathtt 1} = 3$. 

	If $\Gamma(x_0)$ is the perfect matching $2K_2$, the empty graph $N_3$ or the empty graph $N_4$, then we use \cref{alg:mussel} to classify such association schemes. 
	In particular when $\Gamma(x_0)$ is the empty graph $N_4$, we have $k_1 = m_{\mathtt{1}} = 4$ and $a_1 = 0$. 
	So $E_{\mathtt{1}}$ is a light tail and thus $q_{\mathtt{11}}^{\mathtt{1}} = 0$. 
	The algorithm returns the relation-distribution diagrams of the complete bipartite graph $K_{3,3}$, Rook graph $K_3 \square K_3$, Crown graph $\overline{K_2 \square K_5}$ and the tesseract graph $Q_4$. 
	One can easily show that each of them are determined by the relation-distribution diagram (or one can directly check Hanaki and Miyamoto's list since their sizes are small).
\end{proof}

\begin{remark}
	When running the algorithm, one may restrict the cosine $\omega_{1,\mathtt{1}}$ in the range $(\frac{1}{2}, 1)$. 
	Otherwise the spherical embedding of the association scheme is a spherical $[-1,\frac{1}{2}]$-code, hence its size is bounded by $24$ by \cref{thm:chorea}. 
	Again such association schemes can be found in \cite{MR1972022}.
\end{remark}

\section{Discussion} \label{sec:Dot}

The classification of Q-polynomial association scheme with $m_{\mathtt 1} = 4$ is yet not complete. 
The strategy of the classification comes from \cite{MR1608326}, briefly first classifying the local structure and then finding the association schemes which carry these structures. 
If we keep the work in this line, then we need to classify all possible local structures, which requires more computation resources. 
More specifically, we need to classify the spherical $s$-distance set in $\mathbb{R}^3$ where each distance among the points forms a regular graph. 
Some computation in this direction is done in \cite{2018arXiv180406040S}. 
The other difficulty comes from the number of optional new vertices in each step of the iteration. 
When there are many new vertices, we don't have enough equations to determine them. 
Then we have to exhaust all possible cosines. 
The number of possible cosines is related to the valency $v$ of the relation. 
If there are $t$ free variables, then we need to exhaust $v^t$ cases, which grows exponentially. 

We establish an algorithm to generate feasible relation-distribution diagram together with the cosines. 
It may be useful to other classification problems in the study of association schemes. 
It is most efficient when the maximum degree in the relation-distribution diagram is small. 

Note that every case in \cref{thm:Comprehension} is in fact metric rather than merely partially metric. 
This leads us to the following questions.
\begin{question}
	Does there exists an integer $t \geq 2$ such that every $t$-partially metric Q-polynomial association scheme is metric?
\end{question}
\begin{question}
	Does there exists an integer $t \geq 2$ such that every $t$-partially cometric P-polynomial association scheme is cometric?
\end{question}

\section*{Acknowledgements}
This research is guided by Eiichi Bannai. 
The author would like to thank Eiichi Bannai, Peter Cameron, Tatsuro Ito, Jack Koolen, William Martin and Ferenc Sz\"{o}ll\H{o}si for useful discussions. 
This work was supported in part by NSFC [Grant No. 11271257, 11671258] and STCSM [Grant No. 17690740800].

\bibliographystyle{plain}
\bibliography{partiallymetric}

\begin{thebibliography}{10}

\bibitem{MR3281131}
S.~Bang, A.~Dubickas, J.~H. Koolen, and V.~Moulton.
\newblock There are only finitely many distance-regular graphs of fixed valency
  greater than two.
\newblock {\em Adv. Math.}, 269:1--55, 2015.

\bibitem{MR2212140}
Eiichi Bannai and Etsuko Bannai.
\newblock On primitive symmetric association schemes with {$m_1=3$}.
\newblock {\em Contrib. Discrete Math.}, 1(1):68--79, 2006.

\bibitem{MR882540}
Eiichi Bannai and Tatsuro Ito.
\newblock {\em Algebraic combinatorics. {I}}.
\newblock The Benjamin/Cummings Publishing Co., Inc., Menlo Park, CA, 1984.
\newblock Association schemes.

\bibitem{MR1608326}
Eiichi Bannai and Attila Sali.
\newblock On association schemes of balanced property with {$m_1=4$}.
\newblock {\em S\=urikaisekikenky\=usho K\=oky\=uroku}, (991):80--92, 1997.
\newblock Group theory and combinatorial mathematics (Japanese) (Kyoto, 1996).

\bibitem{bannai2017spherical}
Eiichi Bannai and Da~Zhao.
\newblock Spherical embeddings of symmetric association schemes in
  3-dimensional euclidean space.
\newblock {\em Graphs and Combinatorics}, 2019.

\bibitem{MR850954}
N.~L. Biggs, A.~G. Boshier, and J.~Shawe-Taylor.
\newblock Cubic distance-regular graphs.
\newblock {\em J. London Math. Soc. (2)}, 33(3):385--394, 1986.

\bibitem{MR1002568}
A.~E. Brouwer, A.~M. Cohen, and A.~Neumaier.
\newblock {\em Distance-regular graphs}, volume~18 of {\em Ergebnisse der
  Mathematik und ihrer Grenzgebiete (3) [Results in Mathematics and Related
  Areas (3)]}.
\newblock Springer-Verlag, Berlin, 1989.

\bibitem{MR1701280}
A.~E. Brouwer and J.~H. Koolen.
\newblock The distance-regular graphs of valency four.
\newblock {\em J. Algebraic Combin.}, 10(1):5--24, 1999.

\bibitem{MR826793}
Dominique Buset.
\newblock Locally polyhedral graphs.
\newblock In {\em Finite geometries ({W}innipeg, {M}an., 1984)}, volume 103 of
  {\em Lecture Notes in Pure and Appl. Math.}, pages 23--25. Dekker, New York,
  1985.

\bibitem{MR0485471}
P.~Delsarte, J.~M. Goethals, and J.~J. Seidel.
\newblock Spherical codes and designs.
\newblock {\em Geometriae Dedicata}, 6(3):363--388, 1977.

\bibitem{MR1972022}
A.~Hanaki and I.~Miyamoto.
\newblock Classification of association schemes of small order.
\newblock volume 264, pages 75--80. 2003.
\newblock The 2000 ${\rm{C}}om^2MaC$ Conference on Association Schemes, Codes
  and Designs (Pohang).

\bibitem{MR325420}
D.~G. Higman.
\newblock Coherent configurations. {I}.
\newblock {\em Rend. Sem. Mat. Univ. Padova}, 44:1--25, 1970.

\bibitem{MR398868}
D.~G. Higman.
\newblock Coherent configurations. {I}. {O}rdinary representation theory.
\newblock {\em Geometriae Dedicata}, 4(1):1--32, 1975.

\bibitem{MR2755088}
Jianmin Ma and Kaishun Wang.
\newblock Four-class skew-symmetric association schemes.
\newblock {\em J. Combin. Theory Ser. A}, 118(4):1381--1391, 2011.

\bibitem{MR3134270}
Jianmin Ma and Kaishun Wang.
\newblock Nonexistence of exceptional 5-class association schemes with two
  {$Q$}-polynomial structures.
\newblock {\em Linear Algebra Appl.}, 440:278--285, 2014.

\bibitem{2017BITMartin}
William Martin.
\newblock On the connectivity of a basis relation in a symmetric association
  scheme, 2017.

\bibitem{MR2494443}
William~J. Martin and Jason~S. Williford.
\newblock There are finitely many {$Q$}-polynomial association schemes with
  given first multiplicity at least three.
\newblock {\em European J. Combin.}, 30(3):698--704, 2009.

\bibitem{MR1665972}
Markus Meringer.
\newblock Fast generation of regular graphs and construction of cages.
\newblock {\em J. Graph Theory}, 30(2):137--146, 1999.

\bibitem{MR2415397}
Oleg~R. Musin.
\newblock The kissing number in four dimensions.
\newblock {\em Ann. of Math. (2)}, 168(1):1--32, 2008.

\bibitem{MR1609873}
Hiroshi Suzuki.
\newblock Association schemes with multiple {$Q$}-polynomial structures.
\newblock {\em J. Algebraic Combin.}, 7(2):181--196, 1998.

\bibitem{2018arXiv180406040S}
Ferenc {Sz{\"o}ll{\H{o}}si} and Patric R.~J. {{\"O}sterg{\r{a}}rd}.
\newblock {Constructions of maximum few-distance sets in Euclidean spaces}.
\newblock {\em arXiv e-prints}, page arXiv:1804.06040, April 2018.

\bibitem{MR1701285}
Edwin~R. van Dam.
\newblock Three-class association schemes.
\newblock {\em J. Algebraic Combin.}, 10(1):69--107, 1999.

\bibitem{MR3772733}
Edwin~R. van Dam, Jack~H. Koolen, and Jongyook Park.
\newblock Partially metric association schemes with a multiplicity three.
\newblock {\em J. Combin. Theory Ser. B}, 130:19--48, 2018.

\bibitem{MR635704}
Andrew Vince.
\newblock Locally homogeneous graphs from groups.
\newblock {\em J. Graph Theory}, 5(4):417--422, 1981.

\bibitem{MR1635554}
Norio Yamazaki.
\newblock On symmetric association schemes with {$k_1=3$}.
\newblock {\em J. Algebraic Combin.}, 8(1):73--105, 1998.

\end{thebibliography}

\end{document}